\documentclass[12pt]{amsart}[12]
\usepackage{amssymb}  
\usepackage{amscd}
\usepackage{amsfonts}
\usepackage{amsmath}
\usepackage{amsthm}
\usepackage{times}
\usepackage{graphics}
\usepackage{verbatim}
\usepackage[latin1]{inputenc}
\usepackage{subfigure}
\newtheorem{prop}{Proposition}[section]
\newtheorem{thm}[prop]{Theorem}
\newtheorem{cor}[prop]{Corollary}
\newtheorem{lem}[prop]{Lemma}

\newtheorem{prob}[prop]{Problem}

\newtheorem{rem}[prop]{Remark}
\newtheorem{example}[prop]{Example}

\numberwithin{equation}{section}

\newcommand{\C}{\mathbb{C}}

\newcommand{\F}{\mathbb{F}}
\newcommand{\Z}{\mathbb{Z}}

\newcommand{\recursionRc}[3]{{R_{#2,#3}=(q-1) R_{#2,#1 #3} + q R_{#1 #2, #1 #3}}}

\newcommand{\LL}{\Lambda}

\newcommand\Des{{\mathrm{Des}}}

\begin{document}

\title{ $R$-polynomials of finite monoids of Lie type}

\author{K\"ur\c{s}at Aker,\\ Mah\.ir B\.ilen Can,\\ M\"uge Ta\c{s}k{\i}n.}
\address{Feza G\"ursey Institute, Tulane University, Bosphorus University.}
\email{aker@gursey.gov.tr, mcan@tulane.edu, mtaskin@mathstat.yorku.ca}

\begin{abstract}This paper concerns the combinatorics of the orbit Hecke algebra associated with the orbit of a two sided Weyl group action on the Renner monoid of a finite monoid of Lie type, $M$. It is shown by Putcha in \cite{Putcha97}  that the Kazhdan-Lusztig involution (\cite{KL79}) can be extended to the orbit Hecke algebra which enables one to define the $R$-polynomials of the intervals contained in a given orbit. Using the $R$-polynomials, we calculate the M\"obius function of the Bruhat-Chevalley ordering on the orbits. Furthermore, we provide a necessary condition for an interval contained in a given orbit to be isomorphic to an interval in some Weyl group. 
\end{abstract}

\maketitle

\section{Introduction}

Let $G=G(\F_q)$ be a finite group of Lie type (see \cite{Carter}) , $B\subseteq  G$  a Borel subgroup, $T\subseteq B$ a maximal torus, $W$ the Weyl group of $T$ and $S$ the set of simple reflections for $W$ corresponding to $B$. Set 
\begin{equation}\label{E:epsilon}
\epsilon = \frac{1}{|B|}\sum_{g\in B} g.
\end{equation}  
By a fundamental theorem of Tits, it is known that the algebra $\epsilon \C[G] \epsilon$ is isomorphic to the group algebra $\C[W]$. 

The generic Hecke algebra $\mathcal{H}(W)$ of $G$, which  is a deformation of the group algebra $\C[W]$, is a fundamental tool in combinatorics, geometry and the representation theory of $G$. As a $\Z[q^{1/2},q^{-1/2}]$-algebra, $\mathcal{H}(W)$ is generated by a set of formal variables $\{T_w\}_{w\in W}$ indexed by the Weyl group $W$ and obeys a corresponding multiplication rule. 

In \cite{Sol90}, Solomon introduces the first example of a Hecke algebra for monoids, in the case of the monoid ${\mathbf M}_n(\F_q)$, $n\times n$ matrices over a finite field $\F_q$. In  a series of papers (\cite{Putcha93}, \cite{Putcha97}, \cite{Putcha99}), Putcha extends the theory of  Hecke algebras of matrices to all finite regular monoids. In particular, he defines the orbit Hecke algebra 
${\mathcal H}(J)$ for a $\mathcal{J}$-class $J$ in a finite regular monoid.

Let $M$ be a finite monoid of Lie type and $J$ a $\mathcal{J}$-class in $M$. Finite monoids of Lie type are  regular monoids, and the $\mathcal{J}$-class $J=GeG$ in $M$ is a $G\times G$-orbit in $M$ of the subgroup $G\subseteq M$ of invertible elements, where $e$ is an idempotent of $M$. Here, $G$ is a finite group of Lie type. (We use the notation $\mathcal{H}(e)$ in place of $\mathcal{H}(J)$.)

The generic orbit Hecke algebra $\mathcal{H}(e)$ for $GeG$ is what $\mathcal{H}(W)$ is for $G$. In other words, $\mathcal{H}(e)$ is 
a deformation of the contracted semigroup algebra $\epsilon \C_0[J^0] \epsilon$ (the zero of the algebra is the zero of $J^0=J\cup \{0\}$), and $\epsilon$ is as  in (\ref{E:epsilon})).

In \cite{Putcha97}, Putcha extends the Kazhdan-Lusztig involution 
$\overline{T_w} = T^{-1}_{w^{-1}}$ to the generic orbit Hecke algebras.  Using this involution, he defines analogues of $R$-polynomials and $P$-polynomials of \cite{KL79}.  In this article, we investigate the combinatorial properties of these
$R$-polynomials. We show that given an interval contained in an orbit $WeW$ inside the Renner monoid $R$ of the monoid $M$ (\cite{Renner86}),  the constant term of the corresponding $R$-polynomial equals the value of the M\"obius function on the given interval. Using this observation, we give a criterion for when a subinterval of $WeW$ can be embedded into a Weyl group as a subinterval. 


\section{Background}

The monoids of Lie type are introduced and classified by Putcha in \cite{Putcha89} and \cite{Putcha94}.  Among the important examples of these monoids are the finite reductive monoids, \cite{Renner95}. We begin with the notation of reductive monoids. For more information, interested readers may consult \cite{Renner04} and \cite{Putcha88}. For an easy introduction on reductive monoids we especially recommend the expos\'e by L. Solomon \cite{Sol95}.

Let $K$ be an algebraically closed field. An algebraic monoid over $K$ is an irreducible variety $M$ such that the product map is a morphism of varieties.  The set $G=G(M)$ of invertible elements of $M$ is an algebraic group. If $G$ is a reductive group, $M$ is called a \emph{reductive monoid}.

Let $B\subseteq G$ be a Borel subgroup, $T\subseteq B$ a maximal torus, $W=N_G(T)/T$ the Weyl group of the pair $(G,T)$ and $S$ the set of simple reflections for $W$ corresponding to $B$, $\ell$ and $\leq$ the length function
and the Bruhat-Chevalley order corresponding to $(W,S)$.

Recall that the Bruhat-Chevalley decomposition 
\begin{equation*}
G= \bigsqcup_{w\in W} B\dot{w}B,\ \text{for}\ w=\dot{w}T \in N_G(T)
\end{equation*}
of the reductive group $G$ is controlled by the Weyl group $W$ of $G$, where $\dot{w}$ is any coset representative of $w\in W$.

In a reductive monoid $M$, the Weyl group $W$ of the pair $(G,T)$ 
and the set $E(\overline{T})$ 
of idempotents of the embedding $\overline{T} \hookrightarrow M$ form a finite inverse semigroup
$R=\overline{N_G(T)}/T \cong W\cdot E(\overline{T})$ with the 
unit group $W$ and the idempotent set $E(R)=E(\overline{T})$.
The inverse semigroup $R$, called the \emph{Renner monoid} of $M$, governs the Bruhat decomposition
of the reductive monoid $M$:
\begin{eqnarray*}
M= \bigsqcup_{r\in R} B\dot{r}B,\ \text{for}\ r=\dot{r}T \in \overline{N_G(T)}.
\end{eqnarray*}

Recall that the Bruhat-Chevalley order for $(W,S)$ is defined by
\begin{equation*}
x \leq y\ \mbox{iff}\ BxB \subseteq \overline{ByB}.
\end{equation*}
Similarly, on the Renner monoid $R$ of a reductive monoid $M$, the Bruhat-Chevalley order is defined by 
\begin{equation}\label{E:BRordering}
\sigma \leq \tau\ \mbox{iff}\ B\sigma B \subseteq \overline{B\tau B}.
\end{equation}
Observe that the poset structure on $W$ induced from $R$ 
agrees with the original Bruhat poset structure on $W$.

Let $E(\overline{T})$ be the set of idempotent elements in the Zariski closure
of the maximal torus $T$ in the monoid $M$. Similarly, denote  the set of idempotents in the monoid $M$ by $E(M)$. 
One has $E(\overline{T}) \subseteq E(M)$. There is a canonical partial
order $\leq$ on $E(M)$ (hence on $E(\overline{T})$) defined by  
\begin{equation}\label{E:crossorder}
e\leq f\ \Leftrightarrow   \ ef=e=fe.
\end{equation}
Note that $E(\overline{T})$ is invariant under the conjugation action of the Weyl group $W$. We call a subset $\Lambda \subseteq E(\overline{T})$ 
a \emph{cross-section lattice} if $\LL$ is a set of representatives 
for the $W$-orbits on $E(\overline{T})$ 
and the bijection $\LL \rightarrow G \backslash M / G$ defined by 
$e \mapsto GeG$ is order preserving. Then, $\LL=\LL(B) = \{ e\in E(\overline{T}):\ Be=eBe\}$.  

The decomposition $M= \bigsqcup_{e\in \LL} GeG$,
of a reductive monoid $M$ into its $G\times G$ orbits, has a counterpart for the Renner
monoid $R$ of $M$. Namely, the finite monoid $R$ can be written as a disjoint union 
\begin{equation}\label{E:RennerDecomposition}
R = \bigsqcup_{e\in \LL} WeW
\end{equation}
of $W\times W$ orbits, parametrized by the cross-section lattice $\Lambda$.

For $e\in \Lambda$, define
\begin{eqnarray*}
W(e) &:=& \{x\in W:\ xe=ex\}, \\
W_e  &:=& \{x\in W: xe=e \} \trianglelefteq W(e).
\end{eqnarray*}
Both $W(e)$ and $W_e$ are parabolic subgroups of $W$. 

By $D(e)$ and $D_e$, denote 
the minimal coset representatives of $W(e)$ and $W_e$ respectively:
\begin{eqnarray*}
D(e) &:=&  \{x\in W:\ x\ \text{is of minimum length in}\ xW(e) \},\\
D_e &:=& \{x\in W:\ x\ \text{is of minimum length in}\ xW_e \}.
\end{eqnarray*}

Any given element $\sigma \in WeW$ has  {\em the standard form} $x e y^{-1}$
for unique $x$ and $y,$ where $x\in D_e$, and $y\in D(e)$ and $\sigma=x e y^{-1}$.

The length function for $R$ with respect to $(W,S)$ is defined as follows:

Let $w_0$ and $v_0$ be the longest elements of $W$ and $W(e)$ respectively. 
Then $w_0v_0$ is the longest element of $D(e)$.  Set 
\begin{eqnarray*}
\ell(e) & := & \ell(w_0 v_0) = \ell(w_0) - \ell(v_0) \; \text{and}\\ 
\ell(\sigma) & := & \ell(x) + \ell(e) - \ell(y).  
\end{eqnarray*}
Note that $\ell$, in general, need not be equal to the rank function on the graded poset $(R,\leq)$. However, when restricted to a $\mathcal{J}$-class $WeW\subseteq R$ it is equal to the rank function on the induced poset $(WeW,\leq)$.  

By \cite{PPR97}, we know that given two elements $\theta$ and $\sigma$ in the standard form 
$\theta = u e v^{-1} \in WeW$ and $\sigma=xfy^{-1} \in WfW$, 
\begin{equation}
\theta \leq \sigma\ \Longleftrightarrow\ e \leq f,\ u \leq xw,\ yw \leq v 
\qquad \text{for some} \; w\in W(f)W_e.
\end{equation}

\begin{rem}
More generally, let $M$ be a finite monoid of Lie type, and $G\subseteq M$ its group of invertible elements. It is shown by Putcha, \cite{Putcha89} that $M$ has a Renner monoid $R$ as well as a cross section lattice $\Lambda \subset E(M)$. Furthermore,  all of what is said above is true for $R$ and $\Lambda$ of $M$. 
\end{rem}

\section{Orbit Hecke Algebras}

Let $M$ be a finite monoid of Lie type. We use the notation of the previous section.

In \cite{Sol90}, Solomon constructs the Hecke algebra 
and 
the generic Hecke algebra
for the monoid $\mathbf{M}_n(\F_q)$, $n\times n$ matrices over the finite field $\F_{q}$ with $q$ elements.

Until the end of the section, we let $q$ be an indeterminate instead of a prime power. 
Following  Solomon's construction in \cite{Sol90}, the \emph{ generic} Hecke algebra $\mathcal{H}(R)$ of the Renner monoid of $M$ is defined as follows:

The generic Hecke algebra $\mathcal{H}(R)$ is a $\Z[q^{-\frac{1}{2}}, q^{\frac{1}{2}}]$-algebra generated by a formal basis $\{A_{\sigma}\}_{\sigma \in R}$ with respect to multiplication rules

\begin{equation}\label{hecke.structure}
\begin{aligned} 
A_sA_{\sigma}&=
\begin{cases}  A_{\sigma}&\text{if} ~\ell(s\sigma)=\ell(\sigma)\\
A_{s\sigma}&\text{if} ~\ell(s\sigma)=\ell(\sigma)+1\\
q^{-1}A_{s\sigma}+(1-q^{-1})A_{\sigma}&\text{if} ~ \ell(s\sigma)=\ell(\sigma)-1\\
\end{cases} \\
A_{\nu}A_{\sigma}&=A_{\nu \sigma}
\end{aligned}
\end{equation}
for $s\in S,\ \sigma,\ \nu \in R$, where $\ell(\nu)=0$. The products $A_{\sigma} A_s$ are defined similarly.

Fix $e\in \Lambda$ and let $\mathcal{I}\subseteq \mathcal{H}(R)$ be the two-sided ideal
\begin{equation}
\mathcal{I}= \bigoplus_{f< e, \sigma \in WfW} \Z[q^{-\frac{1}{2}},q^{\frac{1}{2}}] A_{\sigma} \subseteq \mathcal{H}(R).
\end{equation}
The \emph{orbit Hecke algebra} 
\begin{equation}
\mathcal{H}(e) = \bigoplus_{\sigma \in WeW} \Z[q^{-\frac{1}{2}},q^{\frac{1}{2}}] A_{\sigma}
\end{equation}
is an ideal of $\mathcal{H}(R)/\mathcal{I}$.  
Note that, if the idempotent $e\in \Lambda$ is the identity element $\mathrm{id}\in W$ of the Weyl group, then the  generic orbit Hecke algebra $\mathcal{H}(e)$  is isomorphic to generic Hecke algebra $\mathcal{H}(W)$ of $G$. 

The algebra $\widehat{\mathcal{H}}(e) := \mathcal{H}(e) + \mathcal{H}(W)$ is called the \emph{augmented orbit Hecke algebra}.

\begin{thm}[\cite{Putcha97}, Theorem 4.1] \label{T:Putcha41}
There is a unique extension of the involution on $\mathcal{H}(W)$ to $\widehat{\mathcal{H}}(e)$ such that for $e \in WeW$ and $\sigma \in WeW$
in standard form $\sigma = s e t^{-1}$,
\[
\begin{array}{lcl}
\overline{A}_{e} & := & 
\sum_{z\in W(e),\ y\in D(e)} \overline{R}_{z,y} A_{zey^{-1}}, \\
\overline{A}_{\sigma} & := & q^{-\ell(t)} \overline{A}_s \sum_{z\in W(e),\ y\in D(e)} \overline{R}_{tz,y} A_{zey^{-1} }.
\end{array}
\]
Here $R_{z,y},\ R_{tz,y} \in \Z[q]$ are $R$-polynomials of $W$. 
\end{thm}

\begin{cor}[\cite{Putcha97},Corollary 4.2] \label{C:Putcha41} 
Let $\sigma \in WeW$. Then there exists $R_{\theta, \sigma} \in \Z[q]$ for 
$\theta \in WeW$, such that in $\widehat{\mathcal{H}}(e)$, 
\begin{enumerate}
\item[(i)] $\overline{A}_{\sigma} = q^{\ell(\sigma)-\ell(e)} \sum_{\theta \in WeW} \overline{R}_{\theta, \sigma} A_{\theta}$,
\item[(ii)] $R_{\theta, \sigma} \neq 0$ only if $\theta \leq \sigma$,
\item[(iii)] $R_{\theta,\theta}=1$.
\end{enumerate}
\end{cor}

In Section \ref{S:Rpolynomials}, we  answer  the following question by Putcha:

\begin{prob} [\cite{Putcha97}, Problem 4.3.] \label{P:Problem1} Determine the polynomials $R_{\theta, \sigma}$ explicitly for $\theta,\sigma \in WeW$. Does $\theta \leq \sigma$ imply $R_{\theta,\sigma} \neq 0$?
\end{prob}

\section{Descent sets for elements of the Renner monoid}

An important ingredient in the study of the combinatorics of the Kazhdan-Lusztig theory for Weyl groups is the descent of an element $w\in W$, which has
been missing in the context of Renner monoids.  
In the following, we extend  the notion of the descent set of an element 
$w\in W$ to a $\mathcal{J}$-class (a $W\times W$-orbit) in the Renner monoid.

Note that for a simple reflection $s$ and $\theta\in R$,
\[
s\theta < \theta  \; (resp., = , >) \qquad \text{if and only if} \qquad 
\ell(s\theta) - \ell(\theta) = -1 \; (resp., 0, 1).
\]

The following lemma can be found in  \cite{Putcha01}.

\begin{lem} 
Let  $I\subset S$, $W_I$ the subgroup generated by $I$, and $D_I$ 
the minimal coset representatives of $W/W_I$ in $W$.
Given $x, y \in D_I$ and  $w, u\in  W_I$.
\begin{enumerate}
\item[(i)] If $xw< yu$ then there exist $w_1, w_2 \in W$ satisfying  $w = w_1w_2$ such that $\ell(w)=\ell(w_1)+\ell(w_2)$, $xw_1 \leq y$ and $w_2\leq u$.
\item[(ii)] If  $wx^{-1}< uy^{-1}$ then there exist $w_1, w_2 \in W$ satisfying $w = w_1 w_2$ such  that  $\ell(w)=\ell(w_1)+\ell(w_2)$, $w_1 \leq u$ and $w_2x^{-1} \leq y^{-1}$.
\end{enumerate}
\end{lem}

\begin{cor}\label{coset.descent}
 We use the notation of the previous Lemma. Let $x\in D_I$ and let $s\in S$.
\begin{enumerate}
\item[(i)] If $x<sx$ then either $sx \in D_I$ or  $sx=xs'$ for some $s' \in W_I\cap S$. 
\item[(ii)] If $sx<x$ then $sx \in D_I$.
\end{enumerate}
\end{cor}

\begin{proof} 

(i) Suppose $sx=x's'$ for some $x' \in D_I$ and $s'\in W_I$. Let $\mathrm{id}$ denote the identity element of the Weyl group. We have 
$$ x \cdot \mathrm{id}=x (\mathrm{id}\cdot \mathrm{id}) \leq x' s'
$$ 
and by previous Lemma it follows that $x \leq x'$ and therefore $l(x)\leq l(x')$. On the other hand  since $$l(x)+1=l(sx)=l(x's')=l(x')+l(s')$$ we have either $x=x'$ and $s' \in I$ or $s'=\mathrm{id}$ and $x'=sx$.

(ii)  If $W$ is finite and  $W_{I}\subseteq W$ then Bj\"{o}rner and Wachs shows in \cite[Theorem 4.1]{BjWs88} that $D_{I}$, which is a generalized quotient, is a lower interval of the weak Bruhat order of $W$. Therefore the result follows.
 \end{proof}

\noindent
For $\sigma \in WeW$, define the left descent set and the right descent of $\sigma$ with respect to $S$ as
\begin{equation*}
\Des_L(\sigma)=\{s\in S ~\mid~ \ell(s \sigma)<\ell(\sigma) \} ~\text{ and }~
\Des_R(\sigma)=\{s \in S ~\mid~ \ell( \sigma s)<\ell(\sigma) \}.
\end{equation*}

\noindent
Then, by the Corollary \ref{coset.descent} we reformulate $\Des_L(\sigma)$ and $\Des_R(\sigma)$ as follows:
\begin{lem}\label{defn.coset.descent}
 Suppose $\sigma \in WeW$ has the standard form $xey^{-1}$ where $x\in D_e$ and $y \in D(e)$. Then $\Des_L(\sigma)=\{ s\in S ~|~ \ell(sx)< \ell(x) \}$, and $ \Des_R(\sigma) =\{ s\in S ~|~\ell(sy)>\ell(y),\ \text{and either}~ sy\in D(e), \text{or}~ sy=ys' ~\text{for some } s'  \in W(e)\cap S ~\text{and}~ \ell(xs')<l(x) \}.$
\end{lem}

\begin{rem} Let $\nu \in WeW$ be the unique element satisfying $\ell(\nu)=0$. Then, it is easy to see that both descent sets of $\nu$ are {\em empty}.  It is essential to emphasize that unlike the usual Weyl group setting,  not every $\sigma \in WeW$ has a left or right descent.  
On the other hand, by using \cite[Theorem 4.1]{BjWs88}, one can show 
the following. 
\end{rem}

\begin{cor} 
For $\sigma \in WeW$ such that $\ell(\sigma)\not=0$ we have 
$\Des_L(\sigma) \cup \Des_R(\sigma)\not=\emptyset$.
\end{cor}

The following example illustrates the possible cases for the descent sets of $\sigma \in WeW$ for $W=S_n$.    

\begin{example} 
Let $\mathbf{M}_4(\F_q)$ be the finite monoid  of $4\times 4$ matrices over the finite field $\F_q$ with $q$ elements. The Renner monoid of $\mathbf{M}_4(\F_q)$ consists of all $4\times 4$ partial permutation matrices\footnote{A partial permutation matrix is a $0-1$ matrix with at most one 1 in each row and each column.}, and its Weyl subgroup is the symmetric group $W=S_{4}$ consisting of permutation matrices. Given a matrix $x=(x_{ij})$ in the Renner monoid, let $(a_1 a_2 a_{3} a_4)$ be the sequence defined by
\begin{equation}\label{E:oneline}
a_j =
\begin{cases}
0,  &\text{if the $j$th column consists of zeros;}\\
i,   &\text{if $x_{ij}=1$.}
\end{cases}
\end{equation}
For example, the sequence associated with the matrix
\begin{equation*}
\begin{pmatrix}
0 & 0 & 0 & 0 \\
0 & 0 & 0 & 0 \\
1 & 0 & 0 & 0 \\
0 & 0 & 1 & 0
\end{pmatrix}
\end{equation*}
is $(3040)$. Let $e$ be the idempotent $e=(1200)\in WeW$ . Then,  $W(e) \cong S_2\times S_2$. The table below illustrates the possible cases for the descent sets for some $\sigma\in WeW$ 
$$\begin{array}{ccc} \underline{\sigma}& \underline{\Des_L}& \underline{\Des_R}\\
\\
(1234)e(3412)=(0012) & \emptyset&\emptyset \\
(1324)e(3412)=(0013) & \{s_2\}&\emptyset \\
(1234)e(1342)=(1002) &\emptyset&\{s_1\} \\
(3214)e(1342)=(3002) &\{s_1,s_2\}&\{s_1\}\\
(4213)e(3124)=(0420) &\{s_1,s_3\}&\{s_2,s_3\}
\end{array}
$$   
\end{example}
Here, $s_{1}=(2134),\ s_{2}=(1324),\ s_{3}=(1243)$ are the simple reflections for $S_{4}$.

\section{$R$-polynomials}\label{S:Rpolynomials}

Given an interval $[\theta,\sigma]\subseteq WeW$ define its length 
$\ell(\theta,\sigma) := \ell(\sigma) - \ell(\theta)$ and its $R$-polynomial
$R_{\theta,\sigma}(q)$ as in Corollary \ref{C:Putcha41}.

\begin{thm}\label{T:Rrecurrence}
Let $\sigma, \theta \in R$ be such that $\ell(\sigma)\not=0$ and  
$\theta \leq \sigma$. Then for $s\in \Des_L(\sigma)$, one has  
\[
R_{\theta,\sigma} = 
\begin{cases}
R_{s\theta, s\sigma}\ &\text{if}\, s\theta < \theta,\\
qR_{\theta,s \sigma}\ &\text{if}\, s\theta=\theta,\\
(q-1)R_{\theta, s \sigma} + q R_{s\theta,s\sigma} \ &\text{if}\, s\theta>\theta.
\end{cases}
\]
Otherwise there exists $s\in \Des_R(\sigma)$ and 
\[
R_{\theta,\sigma} = 
\begin{cases}
R_{\theta s, \sigma s}\ &\text{if}\, \theta  s< \theta,\\
q R_{\theta, \sigma s}\ &\text{if}\, \theta s=\theta,\\
(q-1)R_{\theta, \sigma s} + q R_{\theta s, \sigma s} \ &\text{if}\, \theta s>\theta.
\end{cases}
\]
An addition the above, if $s\theta>\theta$ and $s\sigma=\sigma$, then
$R_{\theta, \sigma} = qR_{s\theta, \sigma}$.
\end{thm}


When $\theta=\sigma$, $R_{\theta,\sigma}(q)=1$. 
If $[\theta,\sigma]$ is an interval 
of length $1$ in $WeW$, then $R_{\theta,\sigma}(q)=q-1$.

\begin{rem} Given two elements $u,v$ of a Weyl group $W$, the polynomial $R_{u,v}(q)\neq 0$ 
if and only if $u\leq v$. If $u\leq v$, then $R_{u,v}(q)$ is a monic polynomial of
degree $\ell(u,v)$ whose constant term is $(-1)^{\ell(u,v)}$. 
\end{rem}
\noindent For the orbit Hecke algebras, we have the following  which 
answers Problem \ref{P:Problem1} \cite{Putcha97}:

\begin{prop} \label{prop:MainPropForPutchaR} 
Let $\theta \leq \sigma \in WeW$. Then $R_{\theta,\sigma}$ is a monic polynomial of degree $\ell(\theta,\sigma)=\ell(\sigma)-\ell(\theta)$ whose the constant term is either 0 or
$(-1)^{\ell(\theta, \sigma)}$.
In particular, $$R_{\theta,\sigma} \neq 0 \qquad \text{ if and only if } \qquad \theta \leq \sigma.$$
\end{prop}

\begin{proof} 
We prove the statement about the constant term. The statement about the degree
and the leading term is proved similarly via induction on $\ell(\sigma)$.

Clearly, the constant term statement holds
 if $\ell(\sigma) \leq 1$. 
As the induction hypothesis, 
we assume that for all pairs $\rho \leq \tau$ in $WeW$ with $\ell(\tau) <\ell(\sigma) $, the constant term of $R_{\rho,\tau}(q)$ is either 0 or $(-1)^{\ell(\rho, \tau)}$. 

Let $\theta \in WeW$ be such that $\theta \leq \sigma$. If $R_{\theta,\sigma}(0)=0$, there is nothing to prove.

Assume that the constant term of $R_{\theta,\sigma}$ is non-zero. Without loss of generality we assume that there exists $s\in \Des_L(\sigma)$. Then, by Theorem \ref{T:Rrecurrence}, we must have either $s\theta > \theta$ or $s \theta < \theta$. 

First suppose that $s\theta < \theta$. 
Then $R_{\theta, \sigma} (q) = R_{s \theta, s\sigma}(q)$. 
Since $\ell(s\theta,s\sigma)=\ell(\theta,\sigma)$, $R_{\theta,\sigma}(0)$ equals $(-1)^{\ell(\theta,\sigma)}$.

Now suppose that $s\theta > \theta$. Therefore, 
$R_{\theta,\sigma}=(q-1)R_{\theta, s \sigma} + q R_{s\theta,s\sigma}$. 
Note that $\theta \leq s\sigma$. 
Consequently,
$R_{\theta,\sigma} (0)=-R_{\theta,s\sigma}(0) $. 
Hence the latter is nonzero,
and by the induction hypothesis it equals $(-1)^{\ell(\theta,s\sigma)}$.
 Therefore, $R_{\theta,\sigma} (0) = (-1)^{\ell(\theta,\sigma)}$ as claimed. 
\end{proof}

\begin{rem}
A very similar line of argument shows that if $R_{\theta, \sigma}(q)$ has a non-zero constant term, then 
$\overline{R_{\theta,\sigma}} = \varepsilon_{\theta} \varepsilon_{\sigma} q_{\theta} q_{\sigma}^{-1} R_{\theta, \sigma}(q)$. However, this equality is false if $R_{\theta, \sigma} (0) = 0$.  
\end{rem}


The lifting property for  Weyl group $W$ states that given $u<v$ in $W$ and 
a simple reflection $s$, if $u>su$ and $v<sv$, then $u<sv$ and $su<v$. We will use 
the lifting property for  $W$ to prove the lifting property for the orbits $WeW$.


\begin{cor}[Lifting Property for $WeW$] \label{LiftingProperty}
Let  $\theta = u e v^{-1}$ and $\sigma = x e y^{-1} $ be in standard form, $\theta < \sigma$ and $s$ be a simple reflection. 

\begin{itemize}
\item[(a)]
If $\theta < s \theta$ and $\sigma < s \sigma$, then $s \theta < s \sigma$.
\item[(b)] 
If $s \theta \geq \theta$ and
$s\sigma \leq \sigma$, then $\theta \leq s \sigma$ and $s \theta \leq \sigma$.
\end{itemize}

\end{cor}

\begin{proof} To make matters short, use  Theorem \ref{T:Rrecurrence} and Proposition \ref{prop:MainPropForPutchaR} to prove (a) or (b) if any of the inequalities is an equality, .
 
What remains to be shown is (b) in the strict case:  $s \theta > \theta$ and $s\sigma < \sigma$.

Because $s\sigma<\sigma$, by Lemma \ref{defn.coset.descent}, $sx<x$ and 
by Corollary \ref{coset.descent} (b),  $sx\in D_e$. Thus, the standard form for $s\sigma$ is $(sx)ey^{-1}$. 

Since $s\theta>\theta$, we observe that $s u \not\leq u$. Otherwise,
$(su)ev^{-1}$ is the standard form of $s\theta$ resulting in a contradiction
$\ell(s\theta)<\ell(\theta)$.
 
As $\theta\leq \sigma$, there is $w\in W(e)$ so that $u\leq xw$ and $yw\leq v$.
First, we prove $\theta \leq s \sigma$.
\begin{itemize}
\item Case $xw \leq s(xw)$: Then $u \leq xw \leq sxw$. Because $(sx)ey^{-1}$ is
the standard form of $s\sigma$, we conclude that $\theta \leq s \sigma$.

\item Case $s(xw) \leq xw$: Apply the lifting property for $W$ to $u$ and $xw$.
So, $u \leq s(xw)=(sx)w$ and again, we get $\theta \leq s \sigma$.
\end{itemize}
 
To prove $s\theta \leq \sigma$, apply (a) to the pair $\theta < s\sigma$.

\end{proof}

\noindent 
Let $q_{\sigma}$ and   $\varepsilon_{\sigma}$ denote, respectively, $q^{\ell(\sigma)}$, and $(-1)^{\ell(\sigma)}$ for $\sigma \in WeW$.


\begin{prop} \label{P:RR=d}
For all $\theta, \sigma \in WeW$,
\begin{equation} \label{eqn:RtimesRtildeEqualsDelta}
\sum_{\theta \leq \nu \leq \sigma}  { R_{\theta,\nu}} 
q_{\sigma} q_{\nu}^{-1} \overline{R_{\nu, \sigma}} = \delta_{\theta, \sigma}. 
\end{equation}
\end{prop}

We call an interval $[\theta,\sigma]$ linear if the interval $[\theta, \sigma]$ is totally ordered with respect to Bruhar-Chevalley order. In this case, 
the interval $[\theta,\sigma]$ has $\ell(\theta,\sigma)+1$ elements.

Using the above proposition together with Proposition \ref{prop:MainPropForPutchaR},
one can classify length $2$ intervals $[\theta,\sigma]\subset WeW$ 
with respect to 
their $R$-polynomials or equivalently, with respect to the constant terms of their
$R$-polynomials:

\begin{equation} \label{table:Length2Intervals}
\begin{array}{|c|c|c|c|c|}
\hline 
\text{Shape of Bruhat Graph} & \text{Number of Elements} & \text{Example:} \; [\theta,\sigma] \subset {\mathbf M}_4(\F_q) & 
R_{\theta,\sigma}(q) & R_{\theta,\sigma}(0) \\
\hline\hline
linear & 3 & (0001)<(0003) & q(q-1) & 0\\
\hline
diamond & 4 & (0012)<(0023) & (q-1)^2 & 1 \\
\hline
\end{array}
\end{equation}

Length $2$ intervals will play a very crucial role later on.


The Mobius function $\mu_{\theta,\sigma}$ corresponding to the interval $[\theta,\sigma]$ is defined by
\begin{equation}
\mu_{\theta,\sigma} :=
\begin{cases}
0 & \, \text{if $\theta \not\leq \sigma$}, \\
1 & \, \text{if $\theta=\sigma$}, \\
-\sum_{\theta \leq \tau < \sigma} \mu_{\theta,\tau} & \, \text{if $\theta < \sigma$}.
\end{cases}
 \end{equation}

\begin{cor} 
For  $\theta,\sigma\in WeW$, \[\mu_{\theta,\sigma}=R_{\theta,\sigma}(0).\]
\end{cor}

\begin{proof}
The equality is clear for $\theta \not\leq \sigma$ and $\theta=\sigma$. For $\theta<\sigma$, 
evaluating (\ref{eqn:RtimesRtildeEqualsDelta})  at $q=0$ yields
$$ R_{\theta,\sigma}(0)=-\sum_{\theta \leq \tau < \sigma} R_{\theta,\tau}(0).$$
Thus, $\mu_{\theta,\sigma}=R_{\theta,\sigma}(0)$.
\end{proof}

Putcha conjectures (Conjecture 2.7 \cite{Putcha04}) the following for reductive monoids, a subclass of monoids
of Lie type:

\[
\mu_{\theta,\sigma}=
\begin{cases}
(-1)^{\ell(\sigma)-\ell(\theta)} & \, \text{if every length $2$ interval in $[\theta,\sigma]$ has $4$ elements,} \\
0 & \, \text{otherwise.} 
\end{cases}
\]

We prove that
\begin{thm} \label{thm:PutchaConjectureAboutMu}
Putcha's conjecture holds for the  $\mathcal{J}$-classes in the monoids of Lie type.
\end{thm}

To prove this theorem, we examine the interplay between
the $R$-polynomial $R_{\theta,\sigma}$, 
the interval it belongs to $[\theta,\sigma]$ and the subintervals $[\alpha,\beta]$
contained in $[\theta,\sigma]$, 
esp. of length $2$ subintervals of $[\theta,\sigma]$.

Next, we prove a relation between $R_{\theta,\sigma}(0)$ and the $R_{\alpha,\beta}(0)$ for a subinterval $[\alpha,\beta]$ of $[\theta,\sigma]$. 

\begin{prop}\label{P:subsabitleri}
Given an interval $[\theta,\sigma]$ with $R_{\theta,\sigma}(0)\neq 0$,
then

\begin{itemize}
\item[(a)] If $s\sigma<\sigma$ 
(or, $s\theta>\theta$, $\sigma s < \sigma$, $\theta s > \theta$) 
for some simple reflection $s$, then for any $\alpha \in [\theta,\sigma]$,
$s\alpha\neq \alpha$.

\item[(b)]
For a subinterval $[\alpha,\beta]$ of $[\theta,\sigma]$ in 
$WeW$, $R_{\alpha,\beta}(0)\neq 0$.
\end{itemize}

\end{prop}

In the proof, we will prove (a) for the case $s\sigma<\sigma$, other cases
being virtually the same.

\begin{proof}
Prove by induction on $\ell(\sigma)$. If $\ell(\sigma)\leq 1$,  then $R_{\theta,\sigma}(q)=(q-1)^{\ell(\theta,\sigma)}$ 
and both statements hold trivially.
 
\noindent{\bf Induction step.}
Note that $s\theta \neq \theta$. If $s\theta<\theta$,
then one can apply induction to $[s\theta,s\sigma]$ as 
$R_{s\theta,s\sigma}(0)=R_{\theta,\sigma}(0)\neq 0$. 
If $s\theta>\theta$,
then one can apply induction to $[\theta,s\sigma]$ as 
$R_{\theta,s\sigma}(0)=-R_{\theta,\sigma}(0)\neq 0$. 

Let $\rho:=\min\{\theta,s\theta\}$. By definition, $s\rho>\rho$ and by lifting
property, for any 
$\alpha \in [\theta, \sigma]$,  $\rho\leq \min\{\alpha, s\alpha\}$. In short, we can apply 
induction to $[\rho,s\sigma]$ since $\ell(s\sigma)<\ell(\sigma)$. 

(a) $R_{\theta,\sigma}(0) \neq 0$ and a simple reflection $s$ 
so that $s\sigma<\sigma$ is given. 
If for all $\alpha\in [\theta,\sigma]$, $s\alpha\neq 0$, there is nothing to prove.

Assume that for some $\alpha \in [\theta, \sigma]$, $s\alpha=\alpha$. 
By lifting property, $\alpha \leq s\sigma$ and hence $\alpha\in [\rho, s\sigma]$. Use induction to conclude that $s\alpha\neq\alpha$ as $s\rho>\rho$. This gives a contradiction. Therefore, we conclude that there is no $\alpha\in [\theta,\sigma]$ such that $s\alpha=\alpha$.

%



(b) Assume that $s\sigma<\sigma$ for some $s\in S$.

Pick any interval $[\alpha,\beta]$ in $[\theta,\sigma]$. If $s\beta>\beta$,
by the lifting property $[\alpha,\beta] \subset [\rho,s\sigma]$ 
(apply the induction step).

Otherwise, $s\beta<\beta$. 

There are two cases:
\begin{itemize}
\item $s\alpha < \alpha$: Then, $R_{\alpha,\beta}(0)=R_{s\alpha,s\beta}(0)$
and $[s\alpha,s\beta] \subset [\rho,s\sigma]$ (apply the induction step).
\item $s\alpha > \alpha$: Then, $R_{\alpha,\beta}(0)=-R_{\alpha,s\beta}(0)$
and $[\alpha,s\beta] \subset [\rho,s\sigma]$ (apply the induction step).
\end{itemize}
\end{proof}

One might wonder if for all \emph{proper} subintervals 
$[\alpha,\beta] \subset [\theta,\sigma]$, then $R_{\alpha,\beta}(0)\neq 0$, is it
true that $R_{\theta,\sigma}(0)\neq 0$? Unfortunately, the answer is no.
As a counterexample, take any linear length $2$ interval $[\theta,\sigma]$.
Then, any proper subinterval $[\alpha,\beta]$ is of length $\leq 1$, hence 
$R_{\alpha,\beta}(0) \neq 0$, yet  $R_{\theta,\sigma}(0) = 0$.

\section{Intervals of Length $\leq 2$}
It is clear that an interval $[\theta,\sigma]$ in $WeW$ with $R_{\theta,\sigma}(0)=0$ can never be embedded into some Weyl group $W'$
as a subinterval $[u,v]$
so that the $R_{\theta,\sigma}(q)$ equals $R_{u,v}(q)$ for the simple
reason that $R_{u,v}(q)=(-1)^{\ell(u,v)}\neq 0$. In fact, more is true. We prove 
the following fact about the Bruhat graphs of the intervals $[\theta,\sigma]$
with $R_{\theta,\sigma}(0)=0$:

\begin{thm} \label{thm:NoEmbedding}
An interval $[\theta, \sigma]$ in $WeW$ cannot be
embedded into any Weyl group $W'$ as a {\em subinterval} if $R_{\theta,\sigma}(0)=0$.
\end{thm}

We prove this assertion by showing that 

\begin{prop} \label{prop:LinearLengthTwoSubIntervalExists}
Given an interval $[\theta,\sigma]$ in $WeW$,  
$R_{\theta,\sigma}(0)=0$ iff 
there exists a linear length $2$  interval $[\alpha,\beta]$ inside $[\theta,\sigma]$.

\end{prop}


Theorem \ref{thm:NoEmbedding} follows from this proposition because of the
well-known fact that
any interval $[u,v]$ of length $2$ in a Weyl group $W'$ is diamond shaped and contains $4$ elements.

\begin{lem} \label{lem:FixedElementImpliesLinearLength2Interval}
Given an interval $[\theta,\sigma]$ in $WeW$,
if there exists a simple reflection $s\in S$ such that $s\sigma<\sigma$ 
(or, $s\theta>\theta$, $\sigma s > \sigma$, $\theta s>\theta$)  
and $s\rho=\rho$  (resp. $\rho s = \rho$) for some $\rho \in [\theta,\sigma]$,
then $[\theta,\sigma]$ contains a linear length $2$ interval.
\end{lem}
\begin{proof} We prove the lemma for $s\sigma<\sigma$ as all other cases are essentially proved the same way.

By assumption, the set $\{\rho : s\rho=\rho\}$ is non-empty. Pick a maximal
element $\alpha$ in this set. Then, for all $\beta\in [\theta,\sigma]$ with
$\alpha < \beta$, $s\beta \neq \beta$.

Now pick $\beta \in [\alpha,\sigma]$ so that it covers $\alpha$.
By the choice of $\alpha$, $s\beta\neq \beta$ as indicated above.
Because $\beta$ covers $\alpha$, $R_{\alpha,\beta}(q)=q-1$. 

If it were that $s\beta<\beta$, then by the recurrence relations, 
Theorem \ref{T:Rrecurrence},
$q-1=qR_{\alpha,s\beta}(q)$, which implies that $\alpha=s\beta$ and $q-1=q$,
both being obvious contradictions.

Therefore, $s\beta > \beta$  and
$R_{\alpha,s\beta}(0)=0$. 
By lifting property, $s\beta \leq \sigma$. The subinterval $[\alpha,s\beta]$ is a linear length
$2$ interval in $[\theta,\sigma]$ as required.
\end{proof}

\begin{lem}
Let $[\theta,\sigma]$ be an interval which contains a linear length $2$ interval $[\alpha,\beta]$. Say for some $s\in S$, $\theta< s\theta$ and $\sigma< s\sigma$. Then, the interval $[s\theta, s\sigma]$ contains
a linear length $2$ interval.
\end{lem}
\begin{proof}
Say $s\alpha \leq \alpha$. By lifting property, $s\theta \leq \alpha < \beta < \sigma < s\sigma$. 
The interval $[\alpha,\beta] \subset [s\theta, s\sigma]$ is the required one.

Otherwise, $s \alpha > \alpha$. Since $[\alpha,\beta]$ is a linear length $2$ interval, $s\beta \geq \beta$.
If $\beta > s\beta$, then $R_{s\alpha,s\beta}=R_{\alpha,\beta}$ and $s\theta \leq s\alpha < s\beta \leq s\sigma$.


If $s\beta=\beta$, then by lifting property $s\theta \leq \beta \leq \sigma < s\sigma$. The results follows by Lemma \ref{lem:FixedElementImpliesLinearLength2Interval}.

\end{proof}

\textbf{Proof of Proposition \ref{prop:LinearLengthTwoSubIntervalExists}.}
($\Leftarrow$) If there is such $\alpha,\beta$, then $R_{\alpha,\beta}(0)=0$
and the rest follows from Proposition \ref{P:subsabitleri}.

($\Rightarrow$) Prove by induction on the length $\ell(\sigma)$.

The base case is $\ell(\theta)=0$ and $\ell(\sigma)=2$ which follows 
from (\ref{table:Length2Intervals}). As usual, assume that $s\sigma<\sigma$
for a simple reflection $s\in S$. 


If $s\theta=\theta$, the result follows by Lemma \ref{lem:FixedElementImpliesLinearLength2Interval}.

If $s\theta > \theta$, then $\recursionRc{s}{\theta}{\sigma}$. Hence $0=R_{\theta,\sigma}(0)=R_{\theta,s\sigma}(0)$.
The length of the interval $[\theta, s\sigma]$ is $\ell(\theta,\sigma)-1$. Apply induction.

If $s\theta<\theta$, then apply the Lemma above and then induction. 
This ends the proof of the Proposition.

\textbf{Proof of Theorem \ref{thm:PutchaConjectureAboutMu}.}
We reiterate what we have already shown.

At this point, we have showed the following are equivalent: For a given interval $[\theta,\sigma]$ in $WeW$,
\begin{enumerate}
\item
$\mu_{\theta,\sigma}=R_{\theta,\sigma}(0)\neq 0$,
\item
$\mu_{\theta,\sigma}=R_{\theta,\sigma}(0)=(-1)^{\ell(\theta,\sigma)}$,
\item 
All length $2$ subintervals of $[\theta,\sigma]$ have $4$ elements
(and are diamond shaped).
\end{enumerate}
Otherwise, $\mu_{\theta,\sigma}=R_{\theta,\sigma}(0)=0$ and $[\theta,\sigma]$ contains a length $2$ subinterval with 
$3$ elements (a linear length $2$ subinterval).

\bibliographystyle{amsalpha}

\begin{thebibliography}{10}
      


\bibitem{BjornerBrenti}
A . Bj\"{o}rner, F. Brenti,
\emph{Combinatorics of Coxeter groups.}
Graduate Texts in Mathematics, 231. Springer, New York (2005).


\bibitem{BjWs82} 
A. Bj\"{o}rner, M. Wachs,
\emph{Bruhat order of Coxeter groups and shellability.}
Adv. in Math. 43 (1982)  87--100. 


\bibitem{BjWs88} 
A. Bj\"{o}rner, M. Wachs,
\emph{Generalized quotients in Coxeter groups.}
Trans. of American Math. Society. 308 (1988)  1--37.




\bibitem{Carter} 
R. W. Carter,
\emph{Finite groups of Lie type: conjugacy classes and complex characters.}
Wiley, New York (1985). 



\bibitem{Humphreys94}
J. E. Humphreys,
\emph{Reflection groups and Coxeter groups.}
Cambridge University Press (1994).

\bibitem{KL79}
D. Kazhdan, G. Lusztig, 
\emph{Representations of Coxeter groups and Hecke
algebras.} {Invent. Math.} {53} (1979), 165--184.





\bibitem{PPR97}
E.A. Pennell, M. Putcha, L. E. Renner, 
\emph{Analogue of the Bruhat-Chevalley order for reductive monoids.} 
Journal of Algebra 196 (1997) 339--368.


\bibitem{Putcha88}
M. Puthca,
\emph{Linear algebraic monoids.} London Math. Soc. Lecture Note Series 133. Cambridge University Press (1988).



\bibitem{Putcha89}
M. Putcha,
\emph{Monoids on groups with BN-pairs.}
Journal of Algebra 120 (1989) 139--169.


\bibitem{Putcha93}
M. Putcha,
\emph{Sandwich matrices, Solomon algebras and Kazhdan-Lusztig polynomials.}
Trans. Amer. Math. Soc. 340 (1993) 415--428. 



\bibitem{Putcha94}
M. Putcha
\emph{Classification of monoids of Lie type.} 
Journal of Algebra 163 (1994) 632--662. 


\bibitem{Putcha97}
M. Putcha
\emph{Monoid Hecke algebras.}
Trans. Amer. Math. Soc. 349 (1997) 3517--3534.


\bibitem{Putcha99}
M. Putcha
\emph{Hecke algebras and semisimplicity of monoid algebras.}
Journal of Algebra 218 (1999) 488--508.


\bibitem{Putcha01}
M. Putcha,
\emph{Shellability in reductive monoids.}
Tran. Amer. Math. Soc. 354 (2001) 413--426.

\bibitem{Putcha04}
M. Putcha,
\emph{Bruhat-Chevalley order in reductive monoids.}
Journal Of Algebraic Combinatorics 20 (2004) 33--53.

\bibitem{Renner86}
L. Renner,
\emph{Analogue of the Bruhat decomposition for algebraic monoids.}
Journal of Algebra 101 (1986) 303--338.


\bibitem{Renner95}
L. Renner,
\emph{Finite reductive monoids.}
Semigroups, Formal Languages and Groups (J. Fountain, Ed.), 381--390, NATO Adv. Sci. Inst. Ser. C Math. Phys. Sci., 466, Kluwer Acad. Publ., Dordrecht (1995).


\bibitem{Renner04}
L. Renner,
\emph{Linear algebraic monoids.}
Encyclopedia of Mathematical Sciences, vol.134,
Subseries: Invariant Theory, vol.5, Springer-Verlag (2005). 






\bibitem{Sol95}
L. Solomon, \emph{An introduction to reductive monoids.} Semigroups, Formal Languages and Groups (J. Fountain, Ed.), 295--352, NATO Adv. Sci. Inst. Ser. C Math. Phys. Sci., 466, Kluwer Acad. Publ., Dordrecht (1995). 


\bibitem{Sol90}
L. Solomon, \emph{The Bruhat decomposition, Tits system and Iwahori ring for the monoid of matrices over a finite field.}
Geom. Dedicata  36  (1990) 15--49.



\end{thebibliography}
 
\end{document}